\documentclass[a4paper,11pt]{article}
 \usepackage{latexsym,amsmath,amssymb,amsthm,graphicx,epsfig,color}
\usepackage{verbatim}
\usepackage{color}
\usepackage[left=2.5cm,right=2cm,top=2cm,bottom=3cm]{geometry}
\usepackage[colorlinks]{hyperref}
\theoremstyle{plain}
\newtheorem{theorem}{Theorem}[section]
\newtheorem{lemma}[theorem]{Lemma}
\newtheorem{proposition}[theorem]{Proposition}
\newtheorem{corollary}[theorem]{Corollary}
\newtheorem{definition}[theorem]{Definition}
\newtheorem{example}[theorem]{Example}
\newtheorem{remark}[theorem]{Remark}

\date{}

\newcommand{\N}{\mathbb{N}} \newcommand{\R}{\mathbb{R}}
\title{Diversities, hyperconvexity and fixed points}
\author{Bo\.zena Pi\c{a}tek and Rafa Esp\'inola}
\begin{document}
\maketitle %Department of Mathematics, Ayatollah Boroujerdi University, Boroujerd, Iran; email:

%\vspace{12pt}\\
%

 \noindent{\bf Abstract.}
Diversities have been recently introduced as a generalization of metrics for which a rich tight span theory could be stated. In this work we take up a number of questions about hyperconvexity, diversities and fixed points of nonexpansive mappings. Most of these questions are motivated by the study of the connection between a hyperconvex diversity and its induced metric space for which we provide some answers. Examples are given, for instance, showing that such a metric space need not be hyperconvex but still we prove, as our main result, that they enjoy the fixed point property for nonexpansive mappings provided the diversity is bounded and that this boundedness condition cannot be transferred from the diversity to the induced metric space.\\

  \maketitle\noindent {\bf Key words}:
Diversities, diversity tight spans, fixed points, hyperconvex metric space, metric tight spans, nonexpansive mappings, phylogenetic.\\
% \noindent {\bf 2000 Mathematics Subject Classification}: 47H10,
%47H09\maketitle

\section{Introduction}
A general theory on {\sl diversities} has recently been proposed by D. Bryant and P. Tupper in \cite{BT}.  The authors introduce diversities in this work as a sort of multi-way metrics which inherits its name after some special appearances in works on phylogenetic and ecological diversity \cite{F, MKH, PS, S}. In \cite{BT} the authors aim to develop a theory of Tight Span Diversities parallel to the theory of Tight Span for metric spaces independently given by A. Dress \cite{D} and J.R. Isbell \cite{isb}. 

As explained by A. Dress, K.T. Huber and V. Moulton in \cite{DHM}, perhaps the first paper that studied metric spaces as such was the work by J. Isbell \cite{isb} on metric tight spans. In this work, J. Isbell finds a natural metric envelop with minimal and uniqueness properties for any given metric space. This metric envelop is named {\sl hyperconvex hull} by J. Isbell, later rediscovered by A. Dress as {\sl metric tight span}, who provides a construction for the unique (up to isometries) minimal hyperconvex metric space where a given metric space may be isometrically embedded. Hyperconvex metric spaces had been introduced some years earlier by N. Aronszajn and P. Panitchpakdi in \cite{AP} as metric spaces which are absolute nonexpansive retracts. Since then a lot has been written on hyperconvex metric spaces, the reader may find a gentle introduction to most of this information in the recent surveys \cite{EK, EL} where hyperconvexity and its connections to existence of fixed points for nonexpansive mappings are explained. These surveys do not deal however with the connection of metric tight spans with phylogenetic problems. For this the reader may consult the delightful exposition on this particular point given in \cite{BT}. 

Motivated by the big impact of tight spans in phylogenetic problems and given that there were some particular natural examples of objects which could be understood as generalized metrics after the name of diversities, D. Bryant and P. Tupper \cite{BT} took up the problem of developing a theory of tight spans for diversities. This project first needed to introduce diversities as a general object which contained the already known examples as particular cases. Then a whole new theory of hyperconvexity for diversities needed to be created. Bryant and Tupper considered these questions by providing us not only with a natural theory of hyperconvex diversities but also showing that a beautifully parallel theory of tight spans existed in this new context. After developing this theory, Bryant and Tupper particularized their approach in the final sections of \cite{BT} for the cases of the so-called diameter and phylogenetic diversities. As a result, these last sections establish very powerful relations between metric tight spans and diversity tight spans of a same metric space when these particular diversities were taken into consideration.

 The work that we present here has been directly motivated by Bryant-Tupper seminal paper on diversities \cite{BT} and references to it will be given throughout our work. Our aim is to find out which general connections may be found between a hyperconvex diversity and its induced metric space. In particular we wonder about the existence of fixed points for nonexpansive self-mappings on such an induced metric space. In the way to give answers to this problem we will need to show new properties and provide examples regarding diversities and induced metric spaces. The work is organized as follows: in Section 2 we recall main facts and definitions from \cite{BT} which are relevant to our discussion as well as main facts on hyperconvex metric spaces which can be found in a more detailed way in any of \cite{EK,EL}. Section 2 is closed with a new fact on the problem of extending nonexpansive mappings from an induced metric space to the diversity. In Section 3 we consider general hyperconvex diversities and want to study which properties the induced metric space inherits.  We prove that this metric space need not be hyperconvex itself and give a sufficient condition that guarantees that this metric space is hyperconvex. As particular examples, we show that both the diameter and the phylogenetic diversities satisfy this condition with respect to their natural induced metrics. As our main result in this section we show that if the diversity is hyperconvex then the induced metric space need not have the fixed point property for nonexpansive mapping even if this metric space is bounded. In Section 4 we remove the boundedness condition from the induced metric space to the diversity to show that, in this case, the induced metric space actually has the fixed point property for nonexpansive mappings. We close the work with a positive result on nonempty intersection of decreasing families of hyperconvex and bounded diversities in the spirit of the one given by J.P. Baillon for hyperconvex metric spaces in \cite{B}.

\section{Preliminaries} 

We begin with metric and hyperconvex metric spaces. Let $(X,d)$ be a metric space, then $\bar B(a,r)$ will stand for the closed ball of center $a\in X$ and radius $r\geq 0$. The {\sl Chebyshev radius} of a set $A\subseteq X$ with respect to $x\in X$ will be, as usual, given by
$$
r_x(A)=\sup_{a\in A}d(x,a).
$$
A subset $A$ of a metric space is said to be admissible if it can be written as an intersection of closed balls. In many aspects admissible subsets of a metric space are a counterpart for convex subset of a linear space. In fact, any subset $A$ of a metric space has an admissible, or ball, hull given by:
$$
{\rm B}(A)=\bigcap\{B:\text{$B$ is a closed ball containing $A$}\}.
$$
Notice that a subset of a metric space is admissible if and only if $A={\rm B}(A)$. Admissible subsets enjoy a number of general properties, the interested reader may check \cite{EK} for more details on this, however we will only need the following representation of admissible sets which is immediate to show. If $A$ is an admissible subset of $X$ then
$$
A=\bigcap_{x\in X}\bar B(x,r_x(A)).
$$

\begin{definition}
\label{hyperconvex}A metric space $M$ is said to be hyperconvex if given any family
$\{x_{\alpha}\}_{\alpha\in\cal A}$ of points of $M$ and any family $\{r_{\alpha}\}_{\alpha\in\cal A}$ of nonnegative
numbers satisfying
\[
d(x_{\alpha},x_{\beta})\leq r_{\alpha}+r_{\beta}%
\]
then 
$$
\bigcap_{{\alpha\in{\cal A}}}B(x_{\alpha},r_{\alpha})\neq\emptyset.
$$
\end{definition}

A particular class of hyperconvex spaces is given by ${\mathbb R}$-trees (or real trees) which will be needed in this work.

\begin{definition}
An $\mathbb{R}$-tree is a metric space $T$ such
that:
\newline\medskip\emph{(i)} there is a unique geodesic
segment (denoted by $\left[  x,y\right]  $) joining each pair of
points $x,y\in T;\medskip$\newline (ii) if $\left[  y,x\right]
\cap\left[  x,z\right]  =\left\{  x\right\}  ,$ then $\left[
y,x\right]  \cup\left[  x,z\right]  =\left[  y,z\right]  .$
\end{definition}

From (i) and (ii) it is easy to deduce:\emph{\medskip}\newline(iii)
\textit{If} $p,q,r\in T,$ \textit{then} $\left[  p,q\right]  \cap\left[
p,r\right]  =\left[  p,w\right]  $ \textit{for some} $w\in M.$

Another notion very relevant in the study of hyperconvex spaces, and especially of metric fixed point theory, is that of nonexpansive mapping.

\begin{definition}
Let $(X,d_1)$ and $(Y,d_2)$ be two metric spaces. A map $T\colon X\to Y$ is said to be nonexpansive if
$$
d_2(Tx,Ty)\le d_1(x,y)
$$
for all $x,y\in X$.
\end{definition}

\begin{definition}
A metric space $(X,d)$ is said to have the fixed point property for nonexpansive mappings if any nonexpansive $T\colon X\to X$ has a fixed point, that is, there exists $x\in X$ such that $Tx=x$.
\end{definition}

A very well known fact, first indepently discovered by R. Sine \cite{S0} and P. Soardi \cite{SO} and then revisited by J.B. Baillon in \cite{B}, is that bounded hyperconvex metric spaces have the fixed point property for nonexpansive mappings. In fact a complete fixed point theory has been developed on hyperconvex spaces since then, the interested reader may check the surveys \cite{EK,EL}. For a more general treatment on metric fixed point theory the reader may check \cite{GK} or for a really exhaustive and more recent monograph \cite{KS}. The reader may also find of interest the following references on fixed points and hyperconvex metric spaces \cite{KKM,S1,S2}.

A very important property of hyperconvex spaces is their relation with injectivity. 

\begin{definition}
A subset $A$ of a metric space $X$ is said to be a {\sl nonexpansive retract} (of $X$) if there exists a nonexpansive retraction from $X$ onto $A$, that is, a nonexpansive mapping $R\colon X\to A$ such that $Rx=x$ for each $x\in A$. $A$ is said to be injective if it is a nonexpansive retract of any metric space where it is isometrically embedded.
\end{definition}

Next we have the announced relation (see \cite{AP,EK} for proofs).

\begin{theorem}\label{inj}
A metric space is hyperconvex if, and only if, it is injective.
\end{theorem}

For a metric space $X$, J. Isbell defined in \cite{isb} the set of
extremal functions $\epsilon (X)$ of $X$ as the set of all
functions $f\colon X\to \R$ such that it satisfies $f(x)+f(y)\geq
d(x,y)$ for all $x$ and $y$ in $X$ and it is pointwise minimal
(see \cite{isb} or \cite[Section 8]{EK} for details). The following theorem shows that
$\epsilon (X)$ can also be regarded as the hyperconvex hull or tight span of $X$.

\begin{theorem}\label{isb}  Let $X$ be a metric space and $\epsilon
(X)$
the set of extremal functions on $X$, then:
\begin{enumerate}
\item $\epsilon (X)$ is a hyperconvex metric space with the metric
$d_{\epsilon (X)}(f,g) =\sup_{x\in X} \vert f(x)-$ $g(x)\vert$.
\item $X$ is isometrically embedded into $\epsilon (X)$ by the
mapping $I_X\colon X\to \epsilon (X)$ defined by $I_X(x)(\cdot )
=d(x,\cdot )$. 
\item If $X$ is isometrically embedded into a hyperconvex space $H$ then $\epsilon (X)$ can also be isometrically embedded into $H$.
\end{enumerate}
\end{theorem}

In \cite{BT} the authors introduce a parallel theory of hyperconvexity, injectivity and construction of diversity tight spans to the one existing for metric spaces. We recall first the notion of diversity.

\begin{definition}
Let $X$ be a set and denote $\langle X\rangle$ as the set of its finite subsets, then a diversity is a pair $(X,\delta)$ where $\delta\colon \langle X\rangle\to \R$ such that 
\begin{enumerate}
\item $\delta (A)\geq 0$, and $\delta (A)=0$ if and only if $|A|\le 1$, where $|A|$ stands for the cardinality of $A$.
\item If $B\neq \emptyset$ then $\delta (A\cup C)\le \delta (A\cup B)+\delta (B\cup C)$.
\end{enumerate}
A diversity will be said bounded if there exists $M\geq 0$ such that $\delta(A)\le M$ for each $A\in \langle X\rangle$.
\end{definition} 

Proofs for statements in the following proposition may be found in \cite{BT}.

\begin{proposition}
Let $(X,\delta)$ be a diversity, then:
\begin{enumerate}
\item $\delta$ is monotone, that is, $\delta (A)\le \delta (B)$ whenever $A\subseteq B$.
\item $\delta$ induces a distance on $X$ defined as $d\colon X\times X\to \R$ given by $d(x,y)=\delta (\{x,y\})$.
\item If $A\cap B\neq \emptyset$ then $\delta (A\cup B)\le \delta (A) +\delta (B)$.
\end{enumerate}
\end{proposition} 

Statements {\it 1} and {\it 3} in the previous proposition will be often applied along this work in the following way: given $A\in \langle X\rangle$ and $z\in X$ then
$$
\delta (A) \le \delta (A\cup \{ z\})\le \sum_{a\in A}\delta (\{ z,a\})=  \sum_{a\in A}d (z,a).
$$
After \cite{BT} we recall the following set of definitions.

\begin{definition}\label{iso}
\begin{enumerate}
\item Let $(X_1,\delta_1)$ and $(X_2,\delta_2)$ be two diversities. A map $\pi\colon X_1\to X_2$ is an embedding if it is one-to-one (injective) and for all $A\in \langle X_1\rangle$ we have that $\delta_1(A)=\delta_2(\pi (A))$, where $\pi(A)=\cup_{a\in A}\pi (a)$.
\item An isomorphism is an onto embedding between two diversities.
\item If $(X,\delta)$ is a diversity, then for each $x$ we define the function $h_x\colon \langle X\rangle\to \R$ by
$$
h_x(A)=\delta (A\cup \{ x\})
$$
for all $A\in \langle X\rangle$. Let $\kappa$ be the map taking each $x\in X$ to $h_x$.
\end{enumerate}
\end{definition}

Natural examples of diversities are provided in \cite{BT}. Next we describe two of them, the diameter and the phylogenetic ones. Both play a relevant role in \cite{BT} and both will be extensively used in our work. 

\begin{enumerate}
\item {\it Diameter diversity.} Let $(X,d)$ be a metric space. For all $A\in \langle X\rangle$ let
$$
\delta(A)=\max_{a,b\in A}d(a,b)={\rm diam}(A).
$$
Then $(X,\delta)$ is a diversity which is called the diameter diversity generated by $(X,d)$. Therefore, any metric space generates a diameter diversity.

\item {\it Phylogenetic diversity.} Consider $(T,d)$ a real tree, let $\mu$ be the one-dimensional Hausdorff measure on it. Notice that in this case $\mu ([a,b])=d(a,b)$ for any $a,b\in T$. If $A\subseteq T$ then the convex hull of $A$ is defined as
$$
{\rm conv}(A)=\bigcup_{a,b\in A}[a,b]
$$
and we say that $A$ is convex if $A={\rm conv}(A)$ (see \cite{EK} for details). Then it happens that   
$$
\delta_t(A)=\mu ({\rm conv}(A))
$$
defines a diversity on $T$ which is called the {\it real-tree} diversity $(T,\delta_t)$ for $(T,d)$. Finally, always following \cite{BT}, a diversity $(X,\delta)$ is a {\it phylogenetic diversity} if it can be embedded in a real-tree diversity for some complete real tree $(T,d)$.
\end{enumerate}

We would like to point out a couple of immediate properties of these diversities which will be needed in Theorem \ref{noext}.

\begin{proposition}\label{continuity}
Let $(T,d)$ be a real tree and $(T,\delta_{\rm diam})$ and $(T,\delta_{\rm phyl})$ its corresponding diameter and phylogenetic diversities. Then,
\begin{enumerate}
\item both diversities coincide when applied to a subset contained in a metric segment of $T$, and
\item both diversities are continuous with respect to the distance in the sense that if $A\in \langle T\rangle$ and $(x_n)\subseteq T$ is such that ${\rm dist}(x_n,A)$ converges to $0$ as $n\to \infty$, then $\delta (A\cup \{x_n\})\to \delta (A)$ with $\delta$ either of these diversities.
\end{enumerate}

\end{proposition} 

Hyperconvexity for diversities is defined as follows in \cite{BT}.

\begin{definition}
A diversity $(X,\delta)$ is said to be hyperconvex if for all $r:\langle X\rangle\to \R$ such that
\begin{equation}\label{rhyper}
\delta\left( \bigcup_{A\in {\cal A}} A\right)\le \sum_{A\in {\cal A}} r(A)
\end{equation}
for all ${\cal A}\subseteq \langle X\rangle$ finite, with $r(\emptyset)=0$, there is $z\in X$ such that $\delta (\{z\}\cup Y)\le r(Y)$ for all finite $Y\subseteq X$.
\end{definition}

D. Bryant and P. Tupper \cite{BT} also give a counterpart of injectivity for diversities and show the equivalence relation as in Theorem \ref{inj} but for hyperconvex and injective diversities. Even more, they build diversity tight span for a given diversity with alike properties to metric tight span. For this work we will not require anything about injectivity of diversities but we will use diversity tight spans as a very useful tool to build examples.

\begin{definition}
Let $(X,\delta)$ be a diversity. Let $P_X$ denote the set of all functions $f\colon \langle X\rangle \to \R$ such that $f(\emptyset)=0$ and 
$$
\sum_{A\in {\cal A}}f(A)\geq \delta\left( \bigcup_{A\in {\cal A}} A\right)
$$
for all $\cal A$ finite subset of $\langle X\rangle$. Write $f \preceq g$ if $f(A)\le g(A)$ for all $A\in \langle X\rangle$. The tight span of $(X,\delta)$ is the set $T_X$ of functions in $P_X$ that are minimal under $\preceq$. 
\end{definition}

This definition gives the tight span $T_X$ of a diversity as a set. Elements of $T_X$ are then characterized in the following way.

\begin{theorem}\label{T_X}
Let $f\colon \langle X\rangle \to\R$ such that $f(\emptyset)=0$. Then $f\in T_{X}$ if and only if for all $A\in  \langle X\rangle$,
\begin{equation}\label{T_XX}
f(A)=\sup_{{\mathcal B}\subseteq \langle X\rangle}\left\{ \delta (A\cup \bigcup_{B\in \mathcal B}B) -\sum_{B\in \mathcal B}f(B)\colon |{\mathcal B}|<\infty\right\}.
\end{equation}
\end{theorem}

Next an adequate natural diversity needs to be put on $T_X$. The answer to this problem is also given in \cite{BT}. It is interesting to compare it with Theorem \ref{isb}.

\begin{theorem}\label{divtight}
Let $(X,\delta)$ be a diversity. Define $\delta_T\colon \langle T_X\rangle\to \R$ as the function such that $\delta_T(\emptyset)=0$ and
$$
\delta_T(F)=\sup_{\{A_f\}_{f\in F}}\left\{\delta \left( \bigcup_{f\in F} A_f\right) - \sum_{f\in F}f(A_f)\colon A_f\in \langle X\rangle \text{ for all } f\in F\right\},
$$
for all finite subset $F$ of $\langle X\rangle$. Then
\begin{enumerate}
\item $(T_X,\delta_T)$ is a hyperconvex diversity,
\item function $\kappa$ from Definition \ref{iso} is an embedding from $(X,\delta)$ into $(T_X,\delta_T)$,
\item for all $A\in \langle X\rangle$ and $f\in T_X$,
$$
\delta_T(\kappa (A)\cup \{f\})=f(A),\text{ and } 
$$
\item if there is an embedding from $(X,\delta)$ into another hyperconvex diversity $(Y,\delta_Y)$ then there is an embedding from $(T_X,\delta_T)$ into $(Y,\delta_Y)$.
\end{enumerate}
The pair $(T_X,\delta_T)$ is called the diversity tight span of $(X,\delta)$.
\end{theorem}

For an easier exposition of our work we reformulate $\delta_T(F)$ in the following way.   

 \begin{lemma}\label{lem1}
Let $(X,\delta)$ be a diversity and $(T_X,\delta_T)$ its diversity tight span, then for each $F\in \langle T_X\rangle $ we have
$$
\delta_T (F)=\sup_{\{A_f\}_{f\in F}}\left\{ \delta \left(\bigcup_{f\in F} A_f\right) -\sum_{f\in F}f(A_f)\colon A_f\in\langle X\rangle \text{ with }A_g\cap A_h=\emptyset,\; g\neq h\right\}.
$$

\end{lemma}

\begin{proof} From \eqref{T_XX} it follows that any $f$ in $T_X$ is monotone with respect to set inclusion. Now, given $\{ A_f\}_{f\in F}$ as in the definition of $\delta_T$ it is always possible to take $\{ A^\prime_f\}_{f\in F}$ which is pairwise disjoint such that $A^\prime_f\subseteq A_f$ and still $\displaystyle \bigcup_{f\in F} A_f=\bigcup_{f\in F}A^\prime_f$. Then the monotonicity of each $f$ implies that
$$
 \delta \left(\bigcup_{f\in F} A_f\right) -\sum_{f_l\in F}f(A_f)\le  \delta \left(\bigcup_{f\in F} A^\prime_f\right) -\sum_{f_l\in F}f(A^\prime_f),
$$
which states our lemma.

\end{proof}

A big role in this work will be played by nonexpansive mappings between metric spaces and diversities. The next definition is given in \cite{BT}.

\begin{definition} \label{noexpadiv}
Let $(X_1,\delta_1)$ and $(X_2,\delta_2)$ be two diversities, then a mapping $T\colon X_1 \to X_2$ is said to be nonexpansive in the sense of diversities if $\delta_1 (A)\geq \delta_2 (T(A))$ for any $A\in \langle X\rangle$, where $\displaystyle T(A)=\bigcup_{a\in A}T(a)$.

If $T$ is a self-mapping defined on a diversity $(X,\delta)$, then a fixed point for $T$ with respect to the diversity is a set $F\in \langle X\rangle$ such that $T(F)=F$.
\end{definition}

\begin{remark}
There is still a second natural definition of nonexpansive mapping between diversities that was not considered in \cite{BT} and that we are not going to consider in this work either. A mapping $T\colon \langle X_1\rangle \to \langle X_2\rangle$ is said to be nonexpansive (of type II) in the sense of diversities if $\delta_1 (A)\geq \delta_2 (T(A))$. Notice that this mapping must send singletons to singletons but $T(A)$ need not be, at least formally, the union of each $T(\{ a\})$ with $a\in A$.
\end{remark}

It is clear that any nonexpansive mapping in the sense of diversities induces a nonexpansive mapping relative to the induced metric spaces. That is, if $T\colon X_1 \to X_2$ is nonexpansive with respect to the diversities and $d_1$ and $d_2$ stand for the respective induced metrics, then $d_1(x,y)=\delta_1 (\{ x,y\})\geq \delta_2 (\{ Tx,Ty\})=d_2(Tx,Ty)$ for any $x,y\in X_1$. The first question we consider in this work is related to this fact. We will make the reasoning for self-mappings for simplicity. Consider $(X,\delta)$ a diversity and $T\colon X\to X$ a nonexpansive mapping with respect to the induced metric space. The question is whether its natural extension to the diversity as $T(A)=\bigcup_{a\in A} T(a)$ is nonexpansive in the diversity sense. Due to the properties of hyperconvexity in relation to extension of nonexpansive mappings, this question may even be more natural in the case when the metric space and the diversity are both hyperconvex. We will show next a result in the negative where the diversity and the induced metric space are both hyperconvex.

\begin{theorem}\label{noext}
Let $(X,\delta)$ be a hyperconvex diversity such that its induced metric space $(X,d)$ is hyperconvex too. If $T\colon (X,d)\to (X,d)$ is nonexpansive then $T\colon (X,\delta)\to (X,\delta)$ need not be nonexpansive.
\end{theorem}

\begin{proof}
To prove this result we will construct an adequate example. Let us consider a~real tree 
$$
\bigcup\limits_{x\in\{a,b,c,d,e,f\}}[\theta,x]
$$ 
such that $[\theta,x]\cap[\theta,y]=\{\theta\}$ if $x\neq y$. For the~subtree $X=\mbox{conv}\{a,b,c\}$ we consider diameter diversity $\delta_1$ and for $Y=\mbox{conv}\{d,e,f\}$ -- the phylogenetic one $\delta_2$. Following a case by case study with 27 possible cases, it is not difficult to see that
$$
\delta(A)=\left\{
\begin{array}{ll}
\delta_1(A), & A\in\left<X\right>\\
\delta_2(A), & A\in\left<Y\right>\\
\delta_1((A\cap X)\cup\{\theta\})+\delta_2((A\cap Y)\cup\{\theta\}), & \mbox{otherwise}
\end{array}
\right.
$$
defines a~diversity on $Z=X\cup Y$. Obviously the induced metric coincides with the tree metric on $Z$ after gluing $X$ and $Y$ through $\theta$. We show next that the diversity $(Z,\delta)$ is hyperconvex. Let us consider a~function $r\colon\langle Z\rangle\to [0,\infty)$ for which
$$
\delta\left(\bigcup\limits_{i=1}^n A_i\right)\leq \sum_{i=1}^n r(A_i), \qquad A_i\in\left<Z\right>, n\in\mathbb{N}.
$$
If it is the case that $\delta(A\cup\{\theta\})\le r(A)$ for any $A\in \langle Z\rangle$ then we are done. Otherwise  there is $G\in \left< Z\right>$ for which
\begin{equation}\label{warG}
\delta(G\cup\{\theta\})>r(G).
\end{equation}
Then obviously $G\subseteq [x,\theta]\setminus\{\theta\}$ for a certain $x\in\{a,b,c,d,e,f\}$. If there are $G_1$ and $G_2$ for which the \eqref{warG} holds then it must be the case that both are contained in the same segment $[x,\theta]\setminus\{\theta\}$. For each $A\in \langle [x,\theta]\rangle$ consider the set
\begin{equation}\label{ball}
\{ a\in [x,\theta]\colon \delta (\{a\}\cup A)\le r(A)\}.
\end{equation}
It is easy to see that each of these sets is a ball. Now, recalling the hyperconvexity of $\delta_1$ or $\delta_2$, whatever applies in $[x,\theta]$, we have that the intersection of all those sets is nonempty and so, since we are in a metric segment, a new metric interval which we denote as $[s,t]$ with $t$ the closest point to $\theta$. Moreover, since \eqref{warG} holds for at least one set in $\langle [x,\theta]\rangle$ and because $\delta$ coincides with the diameter diversity on this interval, it follows that $t\neq \theta$. We want to show next that 
$$
\delta(H\cup\{t\})\leq r(H)
$$ 
for any $H\in \left<Z\right>$.

For each $H\in \left<Z\right>$ we may consider $s(H)$ such that $\delta(H\cup\{s(H)\})\leq r(H)$ and $d(t,s(H))$ is minimal. Clearly, if $s(H)\in [y,\theta]$ and $x\neq y$ we obtain two $G_1$ and $G_2$ for which \eqref{warG} holds, a~contradiction. So $s(H)\in[t,\theta]$ for any $H\in\left<Z\right>$. Let us suppose there is $H$ for which $s(H)\neq t$, i.e., 
\begin{equation}\label{H}
\delta(H\cup\{t\})> r(H).
\end{equation}
We consider two cases. The first one for $x\in\{d,e,f\}$ and the second one for $x\in\{a,b,c\}$.
%%%%%%%%%%%%%%%%%%%%%%%%%%
%dwa przypadki
\begin{enumerate}
\item Let $x=f$. Then $r(H)=\delta(H\cup\{s(H)\}) $ and, by construction, we can choose $G\in \left< [x,\theta]\right>$ satisfying \eqref{warG} and so that its corresponding ball in \eqref{ball} does not contain $s(H)$, so, considering the continuity property given in Proposition \ref{continuity},
$$
\delta(H\cup G)=\delta(H\cup\{s(H)\})+d(s(H),t)+\delta(G\cup\{t\})>r(H)+r(G),
$$
which falls in contradiction with the definition of function $r$ and so states our claim for this case.

\item Now let us consider the case of $x=a$. We assume first that $\theta\in H$ and write $H=H_1\cup H_2$ with $H_1\subset X$ and $H_2\subset Y$. Then $\delta(H)=\delta_1(H_1)+\delta_2(H_2)$ and we focus on the diameter diversity on $X$. If $H_1\subset [a,\theta]$, we apply the same reasoning as in the previous case. Otherwise there is $b^\prime\in [\theta,b]$ or $c^\prime\in [\theta,c]$ such that $r(H)-\delta_2(H_2)=\max\{d(s(H),b^\prime),d(s(H),c^\prime)\}$. So suppose that $r(H)-\delta(H_2)=d(s(H),b^\prime)$ and choose $G$ as above such that it satisfies \eqref{warG} and its corresponding ball in \eqref{ball} does not contain $s(H)$, then
    $$
    \delta(G\cup H_1)=\delta(G\cup\{t\})+d(t,s(H))+d(s(H), b^\prime)>r(G)+d(s(H),b^\prime)
    $$
    and
    $$
    \delta(G\cup H)=\delta_1(G\cup H_1)+\delta_2(H_2)=\delta(G\cup H_1)+\delta(H_2)
    $$
    $$
    =\delta(G\cup\{t\})+d(t,s(H))+d(s(H),b^\prime)+\delta(H_2)>r(G)+r(H)
    $$
    what again leads to a contradiction with the definition of function $r$.
    
    The case in which $\theta \notin H$ is reduced to the above one since for an $H$ satisfying \eqref{H} we have that $\delta (H\cup \{ \theta\})=\delta (H)$ and so nothing changes in the argument.
\end{enumerate}

\vskip3mm

Assume now that $d(x,\theta)=1$ for $x\in \{ a,b,c,d,e,f\}$ and define the~mapping $T\colon X\cup Y\to X\cup Y$ as
$$
\begin{array}{cc}
T([a,\theta])=[d,\theta], & T([d,\theta])=[a,\theta]\\
T([b,\theta])=[e,\theta], & T([e,\theta])=[b,\theta]\\
T([c,\theta])=[f,\theta], & T([f,\theta])=[c,\theta]
\end{array}
$$
in such a~way that $\delta(\{T(x),\theta\})=\delta(\{x,\theta\})$. Hence $T$ is an isometry with respect to the induced metric $d$, so it is nonexpansive too.

On the other hand, if we consider the triple $\{a,b,c\}$, then its image under $T$ is equal to the triple $\{d,e,f\}$ and
$$
\delta(\{d,e,f\})=\delta(\{T(a),T(b),T(c)\})=3>2=\delta(\{a,b,c\}),
$$
what implies that $T$ is not nonexpansive with respect to $\delta$ and completes our proof.\end{proof}

\begin{remark}
Easier examples may be built for the same fact if we do not care about the hyperconvexity of the spaces.
\end{remark}

\section{Hyperconvex diversities and induced metric spaces}
The main motivation for this section is that of studying how the hyperconvexity of a given diversity $(X,\delta)$ determines the geometry of the induced metric space $(X,d)$. This problem is far from being well-understood in general although in \cite{BT} very powerful relations were established for diameter and phylogenetic diversities.

We begin with a first example showing that the induced metric space of a hyperconvex diversity need not be hyperconvex.

\begin{example}\label{ex1} Let $X=\{x,y,z\}$ and $\delta:\langle X\rangle \to \R$ the function given by
$$
\delta (A)=|A|-1, 
$$
where $|A|$ stands for the cardinality of $A$. It is straightforward to see that $(X,\delta)$ is a diversity. Consider, following the notation from Section 2, $(T_X,\delta_T)$ its diversity tight span with induced metric space $(T_X,d_T)$. We claim that $(T_X,d_T)$ is not hyperconvex.

As stated in Theorem \ref{divtight}, given $u,v \in X$ they can be seen as elements of $T_X$ through the mapping $\kappa$, that is, as the respective functions $h_u$ and $h_v$ (see Definition \ref{iso}) with the property that
$$
d_T(h_u,h_v)=\delta (\{ u,v\})= 1
$$
for $u\neq v$. Consider now the collection of balls $\{ \bar B(h_u ,1/2)\colon u\in X\}$ in $(T_X,d_T)$. Hence, if we make $r(u)=1/2$ then, for each $u,v\in X$ with $u\neq v$,
$$
d_T(h_u,h_v)=1=r(u)+r(v).
$$
If $(T_X,d_T)$ is hyperconvex then there must exist a point $f\in T_X$ such that $d_T(f,h_u)\le 1/2$ for each $u\in X$. However, from Theorem \ref{divtight} we have that
$$
d_T(f,h_u)=\delta_T(\{f,h_u\})= f(u)\le 1/2
$$
for each $u\in X$, and so
$$
f(x)+f(y)+f(z)\le 3/2 < 2=\delta(\{ x,y,z\}).
$$
Therefore $f$ is not in $T_X$ and our claim is proved.
\end{example}

\begin{remark} 
Notice that in the previous example both the induced metric space and the diversity are bounded.
\end{remark}

In contrast with the above example, next we give a condition relating the induced metric space and the diversity which guarantees the hyperconvexity of the metric space.

\begin{proposition}\label{prophyp}
Let $(X,\delta)$ be a hyperconvex diversity and $(X,d)$ its induced metric space. If for any $A=\{x_1,\ldots,x_n\}\in \langle X\rangle$ we have that
\begin{equation}\label{hypcon}
(|A|-1)\cdot \delta (A)\le \sum_{1\le i<j\le n}d(x_i,x_j),
\end{equation}
then $(X,d)$ is a hyperconvex metric space.
\end{proposition}

\begin{proof}
Let $r\colon X\to [0,+\infty)$ such that $r(x)+r(y)\geq d(x,y)$. It suffices to prove that there is $\bar x\in X$ such that $d(\bar x,x)\le r(x)$. Let us extent $r$ to the whole $\langle X\rangle$ by
$$
r(A)=\sum_{x\in A}r(x).
$$
From the relation given in the statement, 
$$
(|A|-1)\cdot \delta (A)\le \sum_{1\le i<j\le n}d(x_i,x_j)\le \sum_{1\le i< j\le n}(r(x_i)+r(x_j))
$$
$$
=(|A|-1)\cdot \left( \sum_{1\le i\le n} r(x_i)\right)=(|A|-1)r(A).
$$
Therefore, $\delta (A)\le r(A)$ for any $A\in \langle X\rangle$ and we only need to apply the hyperconvexity of the diversity to find such $\bar x$. 

Notice that the hyperconvexity condition does not require the function $r$ to be defined on the whole $X$, however this is not a restriction since any such function defined on a subset $Y$ of $X$ can be extended to the whole $X$. In fact, consider $r$ defined on $Y$ and $z\in X\setminus Y$ then we can extend $r$ to $Y\cup \{ z\}$ as
$$
r(z)=\sup\{ \max\{d(y,z)-r(y), 0\}\colon y\in Y\}.
$$
To see that this supremum is finite, fix $x_0\in Y$, then for $y\in Y$
\begin{align*}
d(y,z)-r(y)&\le d(y,x_0)+d(x_0,z)-r(y)\\
&\le r(x_0)+d(x_0,z).
\end{align*}
Finally, a standard transfinite inductive reasoning allows us to define $r$ on the whole $X$.
\end{proof}

We do not know whether condition \eqref{hypcon} is also necessary, however we will show next that some of the information given in Lemma 4.2 and Theorem 5.9 in \cite{BT} follows from the above proposition.

\begin{proposition}
The following statements hold:
\begin{enumerate}
\item If $(X,d)$ is a metric space and $(X,\delta_{\rm {diam}})$ is its diameter diversity then the pair $(d,\delta_{\rm {diam}})$ satisfies (\ref{hypcon}).
\item If $\delta_{\rm phyl}$ is a phylogenetic diversity on $X$ with induced metric $d$, then the pair $(d,\delta_{\rm phyl})$ satisfies  (\ref{hypcon}).
\end{enumerate}
\end{proposition}

\begin{proof} Take $A=\{x_1,\ldots,x_n\}\in \langle X\rangle$.

{\it 1}. It is obvious that condition (\ref{hypcon}) holds for $A\in \langle X\rangle$ with $|A|\le 2$. Assume that $|A|>2$. Then
$$
(|A|-1)\delta_{\rm diam}(A) = (|A|-1)\max_{x,y\in A} d(x,y)
$$
which, assuming ${\rm diam}(A)=d( x_1, x_2)$,
$$
\le d( x_1, x_2) +\sum_{3\le i\le n} (d(x_1, x_i)+d(x_i,x_2))\le \sum_{1\le i<j\le n}d(x_i,x_j).
$$

{\it 2}. From the definition of phylogenetic diversity, and since relation (\ref{hypcon}) is trivially inherited by subsets, we can assume that $(X,\delta_{\rm phyl})$ is a tree diversity with a real tree $(X,d)$ as induced metric space. Again we only need to consider the case $|A|>2$. We will prove it inductively on the cardinality of $A$. Consider $|A|=n$ and take $A^\prime = A\cup \{ w\}$ with $w\notin A$. Consider $C={\rm conv}(A)$ the convex hull of $A$ in the metric tree $X$. Take
$$
P_C(w)=u
$$ 
the metric projection of $w$ onto $C$ (the closest point from $C$ to $w$ which always exists and is unique, see \cite{EK} for details), then it is very well known that $u$ is actually the gate from $w$ to $C$, that is,
$$
d(w,x)=d(w,u)+d(u,x)
$$
for any $x\in C$. Clearly, $\delta_{\rm phyl}(A) = \delta_{\rm phyl}(A\cup \{ u\})$ and $\delta_{\rm phyl} (A\cup \{ w\})= \delta_{\rm phyl}(A)+d(u,w)$. So, applying induction hypothesis in the next inequality, we have
\begin{align*}
n\delta_{\rm phyl}(A\cup \{ w\})& = (n-1)\delta_{\rm phyl}(A)+\delta_{\rm phyl} (A) +nd(u,w)\\
&\le \sum_{1\le i<j\le n}d(x_i,y_j)+ \left( \sum_{1\le i\le n}d(x_i,u)\right) +nd(u,w)\\
& = \sum_{x,y\in A}d(x,y)+ \left( \sum_{1\le i\le n}d(x_i,w)\right),
\end{align*}
what completes the proof.
\end{proof}

Another tempting question when dealing with hyperconvex diversities is to try to find the hyperconvex structure of the diversity, that is, to reproduce a inner structure alike to the one existing for hyperconvex metric spaces: distinguished subsets, best proximity properties, intersection properties,... A very first question in this regard would be to define what a hyperconvex subset of a hyperconvex diversity is. Maybe one of the most obvious options is the one we take next.

\begin{definition}
Let $(X,\delta)$ be a hyperconvex diversity and $A\subset X$, then $A$ is said to be hyperconvex with respect to the diversity if the induced diversity of $\delta$ on $A$, that is, the diversity $(A,\delta_{|A})$, is hyperconvex as a diversity.
\end{definition}

Under this definition we can show that any admissible subset of $X$ is hyperconvex with respect to the diversity.

\begin{proposition}\label{admissible}
Let $(X,\delta)$ be a hyperconvex diversity and let $A\subseteq X$ be an admissible subset of $X$ with respect to the induced metric, then $(A,\delta_{|A})$ is a hyperconvex diversity.
\end{proposition}

\begin{proof}
For simplicity we will write $(A,\delta)$ for the induced diversity on $A$. Since $A$ is admissible then $A=\bigcap_{i\in I} \bar B(x_i,R_i)$ for a certain family of centers $\{ x_i\}_{i\in I}\subseteq X$ and radii $\{R_i\}_{i\in I}\subseteq [0,+\infty)$. To show that $(A,\delta)$ is hyperconvex we need to consider a function $r\colon \langle A\rangle\to [0,+\infty)$ satisfying \eqref{rhyper} for any finite collection ${\mathcal A}=(A_k)_{k=1}^n$ of elements of $\langle A\rangle$.

Now we define a new function $\bar r:\langle X\rangle\to [0,+\infty)$ by
\begin{equation*}
\bar r(C)= 
\begin{cases} r(C),&  \text{ if $C\subseteq A$ and $C\neq \{x_i\}$ for any $i\in I$,}\\
 \min\{r(\{ x_i\}),R_i\}, &\text{ if $C=\{x_i\}$ with $x_i\in A$, }\\
R_i, &\text{ if $C=\{x_i\}$ with $x_i\notin A$, }\\
\sum_{y\in C} r_y(A), &\text{ otherwise,} 
\end{cases}
\end{equation*}
where $r_y(A)$ is, as usual, the Chebyshev radius of $y$ with respect to $A$. Let us check now that $\bar r$ satisfies the hyperconvexity condition for finite collections $(C_k)_{k=1}^n$ of finite subsets of $X$. For simplicity we assume that for $1\le k\le n_1$ we have that $C_k$ is such that $C_k\subseteq A$ and $C_k\neq \{x_i\}$ for any $i\in I$ or $C_k=\{x_i\}$ such that $\bar r(\{x_i\})=r(\{ x_i\})$, for $n_1< k\le n_2$ we have that $C_k=\{ x_i\}$ for some $i\in I$ such that $\bar r(\{x_i\})=R_i$, and for $n_2< k\le n$ we have that $C_k\cap (X\setminus A)\neq \emptyset.$ Assume $n_1\geq 1$, then, for any $y\in C_1$, property {\it 2} of diversities implies that
\begin{align*}
\delta\left( \bigcup_{k=1}^n C_k\right) & = \delta\left( \bigcup_{k=1}^{n_1} C_k\cup \bigcup_{k=n_1+1}^{n_2} C_k\cup \bigcup_{k=n_2+1}^n C_k\right)\\
& \le \delta \left( \bigcup_{k=1}^{n_1} C_k\cup\{y\}\right) +\sum_{k=n_1+1}^{n_2}\delta (\{x_{i(k)},y\})+\sum_{k=n_2+1}^n \sum_{z\in C_k}\delta(\{y,z\})\\
& \le \sum_{k=1}^{n_1}\bar r(C_k) +\sum_{k=n_1+1}^{n_2}\bar r(C_k)+\sum_{k=n_2+1}^n\bar r(C_k),
\end{align*}
moreover, if $n_1=0$ then $y$ may be taken as any point in $A$,
so we can conclude that there is a point $\bar x\in X$ such that $\delta (C\cup \{\bar x\})\le \bar r(C)\le r(C)$ for any $C\in \langle A\rangle$. At the same time,
$$
d(x_i,\bar x)= \delta (\{x_i,\bar x\})\le \bar r (x_i)\le R_i
$$
for any $i\in I$, what implies that $\bar x\in A$.  
\end{proof}

%
% Result for hyperconvexity of intersections
%

Let $(X,\delta)$ be a diversity and $Y\subseteq X$ nonempty and bounded with respect to the induced metric. Suppose that we have $r\colon \langle Y\rangle \to\R$ such that satisfies \eqref{rhyper} for any finite subset $\mathcal A$ of $\langle Y\rangle$. Extend $r$ to the whole $\langle X\rangle$
\begin{equation*}
r(A)= 
\begin{cases} r(A),&  \text{ if $A\in \langle Y\rangle$,}\\
r(A\cap Y)+\sum_{a\in A\setminus Y}r_a(Y), &\text{ if $A\cap (X\setminus Y)\neq \emptyset$}. 
\end{cases}
\end{equation*}

Consider now a finite subset $\mathcal A$ of $\langle X\rangle$. Then make $I=\{ A\in {\mathcal A}\colon A\subseteq \langle Y\rangle\}$,  $II=\{ A\in {\mathcal A}\colon A\in \langle X\rangle \text{ with } A\cap (X\setminus Y)\neq \emptyset\}$, take $y\in A_1\in I$ (if $I=\emptyset$ then take $y$ any point in $Y$) and proceed as follows,
\begin{align*}
\delta \left( \bigcup_{A\in \mathcal A}A\right) & = \delta \left( (\bigcup_{A\in \mathcal A}A\cap Y)\cup (\bigcup_{A\in \mathcal A}A\cap X\setminus Y)\right)\\
&\le  \delta \left( \bigcup_{A\in \mathcal A}A\cap Y\right) +\delta \left((\bigcup_{A\in \mathcal A}A\cap X\setminus Y)\cup \{y\}\right)\\
&\le \sum_{A\in I}r(A)+\sum_{A\in II}r(A\cap Y)+ \sum_{a\in \bigcup_{A\in II}A\setminus Y} \delta (\{a,y\})\\
& \le \sum_{A\in\mathcal A}r(A).
\end{align*}

Therefore the following statement has been proved.

\begin{lemma}\label{hyperholds}
If $(X,\delta)$, $Y$ and $r$ are as above then it is possible to extend $r$ to the whole $\langle X\rangle$ so that \eqref{rhyper} still holds by the equation
\begin{equation}\label{55}
r(A)= r(A\cap Y)+ \sum_{a\in A\setminus Y}r_a(Y).
\end{equation}
\end{lemma}

The next proposition gives new examples of hyperconvex subsets with respect to a hyperconvex diversity. This result will be needed in Section 4. 

\begin{proposition}\label{prophyper}
Let $(X,\delta)$ be a hyperconvex diversity and $Z\subseteq X$ nonempty and bounded with respect to the induced metric. Let $r\colon \langle Z\rangle\to \R$ satisfying the hyperconvexity condition \eqref{rhyper} on $\langle Z\rangle$ and define
$$
Y=\{ x\in X\colon \delta (A\cup \{x\})\le r(A) \text{ for each $A\in \langle Z\rangle$}\}.
$$
Then $Y$ is nonempty and $(Y,\delta)$ is a hyperconvex diversity.
\end{proposition}

\begin{proof} We consider $r$ as given by Lemma \ref{hyperholds}. Since $(X,\delta)$ is hyperconvex then it is clear that $Y$ is nonempty. Moreover, take $z\in Z$ then, for any $y\in Y$, it must be the case that $\delta(\{ z,y\})\le r(z)$ so $Y$ is also bounded.

We need to show that given $s\colon \langle Y\rangle\to \R$ satisfying \eqref{rhyper} then there exists $y\in Y$ such that $\delta (A\cup\{ y\})\le s(A)$ for each $A\in \langle Y\rangle$. From $s$ and $r$ we define the following function on $\langle X\rangle$,
\begin{equation*}
\bar r(A)= 
\begin{cases} \min\{r(A),s(A)\},&  \text{ if $A\in \langle Y\cap Z\rangle$,}\\
s(A),& \text{if $A\in \langle Y\rangle$ and $A\cap Y\setminus Z\neq \emptyset$,}\\
r(A),&  \text{if $A\in \langle Z\rangle$ and $A\cap Z\setminus Y\neq \emptyset$,}\\
\sum_{a\in A}r_a(Y), &\text{otherwise}. 
\end{cases}
\end{equation*}
Notice that $\bar r(A)\le r(A)$ for $A\in \langle Z\rangle$ and $\bar r(A)\le s(A)$ for $A\in \langle Y\rangle$. We want to see that $\bar r$ satisfies \eqref{rhyper} on $\langle X\rangle$. Consider ${\mathcal A}\subseteq \langle X\rangle$ finite with ${\mathcal A}=\{ A_1,\ldots, A_{n_1},A_{n_1+1},\ldots, A_{n_2}, A_{n_2+1},\ldots, A_n\}$ with $\bar r (A_i)=s(A_i)$ for $1\le i\le n_1$; $\bar r(A_i)=r(A_i)$ for $n_1+1\le i\le n_2$ and $\bar r(A_i)=\sum_{a\in A}r_a(Y)$ for $i>n_2$. If $n_1\geq 1$ then take $y\in A_1\subseteq Y$ otherwise just make $y$ any point in $Y$ and proceed as follows, with corresponding terms as $0$ if $n_1=n_2$ , $n_2=n$ or $n_1=n$,
\begin{align*}
\delta\left( \bigcup_{A\in \mathcal A} A\right)&\le \delta\left( \bigcup_{1\le i\le n_1} A_i\right) + \delta \left( \bigcup_{n_1<i\le n_2}A_i\cup \{ y\}\right) + \delta \left( \bigcup_{n_2<i\le n}A_i\cup \{ y\}\right)\\
& \le \delta\left( \bigcup_{1\le i\le n_1} A_i\right) + \sum_{n_1<i\le n_2} \delta( A_i\cup \{ y\}) + \sum_{n_2<i\le n} \sum_{a\in A_i}\delta ( \{y,a\})\\
& \le \sum_{1\le i\le n_1} s(A_i)+ \sum_{n_1<i\le n_2} r( A_i) + \sum_{n_2<i\le n} \sum_{a\in A_i}r_a(Y) = \sum_{A\in \mathcal A} \bar r(A_i).
\end{align*}

Now we apply hyperconvexity of $X$ to deduce that there exists $\bar y$ such that $\delta (A\cup \{\bar y\})\le \bar r(A)$ for each $A\in \langle X\rangle$. But now, since $\bar r\le r$ on $\langle Z\rangle$, it must be the case that $\bar y\in Y$ and again, since $\bar r\le s$ on $\langle Y\rangle$, we finally have that $\delta (A\cup\{ \bar y\})\le s(A)$ for any $A\in \langle Y\rangle$ what finishes our proof.

\end{proof}

\begin{remark}
Sets given in Proposition \ref{prophyper} are a natural candidate for admissible subsets of a diversity.
\end{remark}

%
% End of result for hyperconvexity
%

If $(X,\delta)$ is a hyperconvex diversity and $(X,d)$ is its induced metric space then one may wonder about which metric properties can be assured for $(X,d)$. This seems to be a difficult task. So far we know that this space need not be hyperconvex, however, next we show that such a metric space must be complete. Remember that any hyperconvex metric space is complete too (for details see \cite{AP,EK}).

\begin{proposition}
Let $(X,\delta)$ be a hyperconvex diversity and $(X,d)$ its induced metric space, then $(X,d)$ is complete.
\end{proposition}  

\begin{proof}
Let $\{x_n\}$ be a Cauchy sequence in $(X,d)$. Given $x\in X$ we can assume that there exists $N=N(x)\in\N$ such that  $x\neq x_n$ for any $n\geq N(x)$ (otherwise the sequence would be trivially convergent). Since $\{ x_n\}$ is bounded then 
$$
r(x)= \sup_{n\geq N(x)}d(x,x_n)
$$
is a real number. 

For each finite set $\{y_1,y_2,\ldots ,y_n\}\subseteq X$, define 
$$
r(y_1,y_2,\ldots ,y_n)=\sum_{i=1}^n r(y_i).
$$
We want to show that for each such set we have that
$$
\delta(\{y_1,y_2,\ldots ,y_n\})\le r(y_1,y_2,\ldots ,y_n).
$$
Since $\{y_1,y_2,\ldots ,y_n\}$ is finite then there is $m$ such that $m>\max\{ N(y_i)\colon i=1,2,\ldots, n\}$. Therefore, $r(y_i)\geq d(y_i,x_m)$ for $i=1,2,\ldots ,n$. On the other hand, from diversity properties,
$$
\delta (\{ y_1,y_2,\ldots, y_n\})\le \sum_{i=1}^n\delta (\{x_m,y_i\}),
$$
and so $\displaystyle \delta (\{ y_1,y_2,\ldots, y_n\})\le \sum_{i=1}^n r(y_i).$
Now, we can apply that $(X,\delta)$ is a hyperconvex diversity to assure that there exists $\bar x\in X$ such that
$$
\delta (\{ y_1,y_2,\ldots, y_n,\bar x\})\le \sum_{i=1}^n r(y_i),
$$
for any finite collection $\{y_1,y_2,\ldots ,y_n\}$ of points of $X$. From where, in particular, $d(\bar x,x_n)\le r(x_n)\le \sup_{m>n}d(x_n,x_m)$ for each $n\in \N$. Taking limit when $n\to \infty$ and recalling that $\{x_n\}$ is Cauchy completes the proof.
\end{proof}

After Example \ref{ex1} the question of whether a nonexpansive self-mapping from a bounded metric space induced by a hyperconvex diversity has a fixed point is relevant. We close this section with a negative result in this regard.

\begin{theorem}\label{th1}
Let $(X,\delta)$ be a hyperconvex diversity such that its induced metric space $(X,d)$ is bounded. Then $(X,d)$ need not have the fixed point property for nonexpansive mappings.
\end{theorem}

\begin{proof}
We will show that there is such a metric space admiting a nonexpansive self-mapping which is fixed point free. Let $X=\{x_1,x_2,\ldots \}$ an infinity and countable set provided with the diversity $\delta (A)=|A| -1$ for $A\in \langle X\rangle$. Let $(T_X,\delta_T)$ be its diversity tight span. Now for 
each $A\in \langle X\rangle$ define $A^{+}$ as
$$
A^{+}=\{ x_{k+1}\colon x_k\in A\}.
$$
It is obvious that $\delta (A)=\delta (A^{+})$ for any $A\in \langle X\rangle$. For each $f\in T_X$ define $T(f)=\bar f$ as
$$
\bar f(A^{+})=f(A),\; \bar f (A^{+}\cup \{x_1\})=f(A)+1.
$$
We claim that $\bar f\in T_X$. To show this we will check condition \eqref{T_XX} in Theorem \ref{T_X}. Consider first a set $A$ in $\langle X\rangle$ such that $x_1\notin A$, that is, a set of the form $A^+$.  Then, since $f\in T_X$,
\begin{align*}
f(A)&= \sup_{{\mathcal B}\subseteq \langle X\rangle} \left\{ \delta (A\cup \bigcup_{B\in \mathcal B} B) -\sum_{B\in \mathcal B}f(B)\colon\; |{\mathcal B}|\le\infty\right\}\\
&= \sup_{{\mathcal B}\subseteq \langle X\rangle} \left\{ \delta (A^+\cup \bigcup_{B\in \mathcal B} B^+) -\sum_{B\in \mathcal B}\bar f(B^+)\colon\; |{\mathcal B}|\le \infty\right\}\\
& = \bar f (A^{+}).
\end{align*}
Therefore we can rewrite the above as
\begin{equation*}\label{sup}
\bar f (A^{+})= \sup_{{\mathcal B}\subseteq \langle X\setminus \{x_1\}\rangle} \left\{ \delta (A^+\cup \bigcup_{B\in \mathcal B} B) -\sum_{B\in \mathcal B}\bar f(B)\colon\; |{\mathcal B}|\le \infty\right\}.
\end{equation*}
Take ${\mathcal B}$ now without the previous restriction and make $I=\{B\in {\mathcal B}\colon x_1\notin B\}$ and $II=\{B\in {\mathcal B}\colon x_1\in B\}$ where the cardinality of $II$ is $C\geq 1$. Then, for each such $\mathcal B$,
\begin{align*}
\delta (A^+\cup \bigcup_{B\in \mathcal B} B)-\sum_{B\in \mathcal B}\bar f(B) &= \delta (A^+\cup \bigcup_{B\in I} B\cup \bigcup_{B\in II}B)-\sum_{B\in I}\bar f(B)-\sum_{B\in II}\bar f(B)\\
& = \delta (A^+\cup \bigcup_{B\in I} B\cup \bigcup_{B\in II}(B\setminus\{x_1\}))+1-\sum_{B\in I}\bar f(B)-\sum_{B\in II} \bar f(B\setminus\{x_1\})-C\\
& \le \delta (A^+\cup \bigcup_{B\in I} B\cup \bigcup_{B\in II}(B\setminus\{x_1\}))-\sum_{B\in I}\bar f(B)-\sum_{B\in II} \bar f(B\setminus\{x_1\})\\
&\le \bar f(A^+),
\end{align*}
which states our claim for sets $A$ of the form $A^+$. 

Now, if $x_1\in A$ then we can write $A=C^+\cup\{x_1\}$. We know that $\bar f(A)=\bar f(C^+\cup\{x_1\})= f(C)+1=\bar f(C^+)+1$ and following the same line of argument as before the proof of our claim is complete. 

Next make $X_1=T(T_X)$ and consider the induced diversity by $\delta_T$ (denoted again by $\delta_T$) onto $X_1$, we want to show that $T$ is an isomorphism (Definition \ref{iso}) between diversities $(T_X,\delta_X)$ and $(X_1,\delta_T)$. Let $F$ be a finite subset of $T_X$, then, from Lemma \ref{lem1},
\begin{align*}
\delta_T(F) &=\sup_{\{A_f\}_{f\in F}}\left\{ \delta \left(\bigcup_{f\in F} A_f\right) -\sum_{f\in F} f (A_f)\colon A_g\cap A_h=\emptyset,\; g\neq h\right\}\\
& = \sup_{\{A_f\}_{f\in F}}\left\{ \delta \left(\bigcup_{f\in F} A_f^+\right) -\sum_{f\in F} \bar{f} (A_f^+)\colon A_g^+\cap A_h^+=\emptyset,\; g\neq h\right\}\\
&= \sup_{\{A_{\bar{f}}\}_{\bar{f}\in \bar{F}}} \left\{ \delta \left(\bigcup_{\bar{f}\in \bar{F}} A_{\bar{f}}^+\right) -\sum_{\bar{f}\in \bar{F}} \bar{f} (A_{\bar{f}}^+)\colon A_{\bar g}^+\cap A_{\bar h}^+=\emptyset,\; \bar g\neq \bar h\right\}.
\end{align*}
We want to show that the last term gives the actual value of $\delta_T(\bar F)$, where $\bar F=T(F)$. For that it only rests to prove that we still have the same relation if we allow $x_1\in \bigcup_{\bar f\in\bar  F} A_{\bar f}$. Since we can assume these sets are pairwise disjoint, we have that $x_1$ is in exactly one set of the collection which we denote by $A_1$ with corresponding $\bar f$ as $\bar f_1$. Then
$$
\delta \left( \bigcup_{\bar{f}\in \bar{F}} A_{\bar{f}}\right) - \sum _{\bar{f}\in \bar{F}} \bar{f}(A_{\bar{f}})= \delta \left( \bigcup_{\bar{f}\in \bar{F}} A_{\bar{f}}\right) - \sum_{\bar{f}\in \bar{F}\setminus\{ \bar{f}_1\}} \bar{f}(A_{\bar{f}}) - \bar{f}_1(A_{1})
$$
$$
=\delta \left( \bigcup_{\bar{f}\in \bar{F}}A_{\bar{f}}\right) - \sum_{\bar{f}\in \bar{F}\setminus\{ \bar{f}_1\}} \bar{f}(A_{\bar{f}}) - \bar{f}_1(A_{1}\setminus \{x_1\})-1
$$
$$
=\delta \left( \bigcup_{\bar{f}\in \bar{F}} A_{\bar{f}} \setminus \{x_1\} \right) - \sum_{\bar{f}\in \bar{F}\setminus\{ \bar{f}_1\}} \bar{f}(A_{\bar{f}}) - \bar{f}_1(A_{1}\setminus \{x_1\}).
$$
Therefore, adding $\{x_1\}$ to one of the sets does not increase the above supremum and so we have that $\delta_T(F)=\delta_T(\bar F)$. In consequence $T$ is a nonexpansive mapping (actually an isomorphism) between diversities which, in particular, implies that its restriction to respective induced metric spaces is nonexpansive as a self-mapping on a metric space. 

Now set $X_2=T(X_1)$. Any function $f$ in $X_1$ trivially satisfies that
$$
f(A\cup \{x_1\})= f(A)+1
$$
for any $A\in \langle X\rangle$ with $x_1\notin A$. Consider now $f\in X_2\subseteq X_1$, $g\in X_1$ such that $Tg=f$ and $A^+\in \langle\{ x_3,x_4,\ldots \}\rangle$. Then $f(A^+\cup \{x_1\})=f(A^+)+1$,
$$
f(A^+\cup \{x_2\})= Tg(A^+\cup \{ x_2\})= g(A\cup \{x_1\})= g(A)+1=f(A^+)+1,
$$
and 
$$
f(A^+\cup \{x_1\}\cup \{x_2\})= Tg(A^+\cup \{x_1\}\cup\{ x_2\})= g(A\cup \{x_1\})+1= g(A)+2=f(A^+)+2.
$$

Iterating this process we may construct a decreasing sequence of such diversities $(X_n,\delta_T)$ with $X_n=T(X_{n-1})$. Assume now that there exists $f\in \bigcap_{n=1}^\infty X_n$. Then, since $f\in X_n$ for each $n$,
$$
f(\{ x_1\})= f(\emptyset )+1 =1,\ldots\:, f(\{ x_n\})=f(\emptyset) +1.
$$
Take $A\in \langle X\rangle$, then there exists $n\in \N$ such that $A\subseteq \{ x_1,x_2,\ldots ,x_n\}$. Since $f\in X_n$, it must be the case that
$$
f(A)=f(\emptyset)+|A|=|A|.
$$
But then it must also be the case that, if $x_1\in A$,
$$
f(A)=1+|A|-1=1+h_{x_1}(A),
$$
while otherwise
$$
f(A)=|A| =h_{x_1}(A),
$$
where $h_x$ is as in Definition \ref{iso}. Consequently $h_{x_1} <f$ and so $f\notin T_X$. So the intersection of all sets $X_n$ must be empty and $T$, as a nonexpansive self-mapping on the induced metric space, must be fixed point free since any of its fixed point would be in all sets $X_n$.

\end{proof}

In the previous proof we have shown that nonexpansive self-mappings defined on a bounded induced metric spaces by a hyperconvex diversity need not have fixed point. We have done it by showing the failure of another relevant property for hyperconvex spaces, this is the fact that decreasing sequences of bounded hyperconvex spaces have nonempty intersection (firstly proved in \cite{B}). Notice, however, that although the induced metric spaces are bounded the diversities under consideration in our proof are not. In the next section we will study the case when the diversity itself is assumed to be bounded. Now we give some consequences of the previous theorem.

\begin{corollary}
Decreasing sequences of hyperconvex diversities with induced bounded metric spaces do not have the nonempty intersection property in general. 
\end{corollary} 

\begin{remark} Since bounded hyperconvex metric spaces have the fixed point property for nonexpansive mappings, we can deduce the existence of spaces as the one given by Example \ref{ex1} from the above theorem. Notice, however, that Example \ref{ex1} provides such an example where both the metric space and the diversity are bounded.
\end{remark}

\begin{corollary}
The mapping from Theorem \ref{th1} is also fixed point free as a mapping between diversities, that is, there is no $F\in \langle X\rangle$ nonempty such that $T(F)=F$.
\end{corollary}

\begin{proof}
Consider $(T_X,\delta_T)$, $(X_1,\delta_T)$ and $T$ as in the proof of Theorem \ref{th1}. If such a set $F$ exists then we can easily take it to be minimal (that is, it does not have a proper subset which is a fixed point for $T$). Assume further that the cardinality of $F$ is $k$, then we can order its elements in such a way that $T^m(f_1)=T^i(f_1)$ where $m\equiv i\; \text{mod}(k)$ and $i\in \{1,2,\ldots, k\}$. In particular we would have that $f_1$ is in an infinite collection of sets $X_n$ and, since this is decreasing collection, it is in the intersection of all them too contradicting the proof of Theorem \ref{th1}.
\end{proof}

\section{Bounded diversities and fixed points}

In this section we consider the case when the diversity $(X,\delta)$ is bounded, i.e., there exists $M\in \R$ such that $\delta (A)\le M$ for each $A\in \langle X\rangle$. Remember that $B(A)$ stands for the ball hull of a subset $A$ of a metric space $X$. We need to recall the following immediate lemma for ball hulls. 

\begin{lemma}
Let $(X,d)$ be a metric space, $A\subseteq X$ nonempty and $B(A)$ the ball hull of $A$, then $r_x(A)=r_x(B(A))$ for any $x\in X$.
\end{lemma}

In Theorem 5 of \cite{B} it is proved that any bounded hyperconvex metric space has the fixed point property for nonexpansive self-mappings. We present next a counterpart of this result for diversities. Our proof follows the same patterns of that of Theorem 5 in \cite{B} although, of course, technical difficulties arise.

\begin{theorem}\label{fixedpoint}
Let $(X,\delta)$ be a hyperconvex and bounded diversity with induced metric space $(X,d)$, and $T\colon (X,d)\to (X,d)$ a nonexpansive mapping. Then $T$ has a fixed point in $X$. 
\end{theorem}

\begin{proof}
Let us consider
$$
U=\{ A\subseteq X\colon A\neq \emptyset,\; A=B(A),\; T(A)\subseteq A\} 
$$
partially ordered by set-inclusion $\subseteq$. We want to show first that this family satisfies the hypothesis of Zorn's lemma. First, $U$ is nonempty because $X\in U$. Let $(A_i)_{i\in I}$ be a decreasing chain ordered by inclusion. We claim that its intersection is nonempty, that is, $\displaystyle\bigcap_{i\in I}A_i\neq \emptyset$. Notice that for each admissible subset $A$ of $X$ we have 
$$
A=\bigcap_{x\in X} \bar B(x,r_x(A)).
$$

For each $x\in X$ we have that $r_x(A_i)\le r_x(A_j)$ for $i\geq j$, and so we can define
$$
r(x)=\inf \{ r_x(A_i)\colon i\in I\}.
$$
We can assume that $r(x)>0$ since otherwise $x$ is in the intersection of sets $A_i$. Let us consider $\{y_1,y_2,\ldots,y_n\}$ a finite collection of points in $X$ and let $\varepsilon >0$. Then, for each $k\in \{ 1,2,\ldots, n\}$ there exists $i(k)$ such that
$$
r_{y_k}(A_{i(k)})\le r(y_k)+\varepsilon.
$$
We can agree that
$$
A_{i(1)}\subseteq A_{i(2)}\subseteq \ldots \subseteq A_{i(n)}.
$$
Hence
$$
r_{y_k}(A_{i(1)})\le r(y_k)+\varepsilon.
$$
So taking any $a\in A_{i(1)}$ we have that $d(y_k,a)\le r_{y_k}(A_{i(1)})\le r(y_k)+\varepsilon$ and
$$
\delta (\{y_1,y_2,\ldots ,y_n\})\le \sum_{k=1}^n\delta (\{y_k,a\})\le \sum_{k=1}^nr(y_k)+n\varepsilon.
$$
Since this inequality stands for any positive $\varepsilon$, we have that
$$
\delta (\{y_1,y_2,\ldots ,y_n\})\le  \sum_{k=1}^nr(y_k)
$$
for any finite collection $\{y_1,y_2,\ldots,y_n\}$ of points in $X$. Therefore, for a given finite collection $\{x_1,x_2,\ldots, x_m\}$ of elements of $X$ define 
$$
r(\{x_1,x_2,\ldots , x_m\})=\sum_{k=1}^mr(x_k)
$$ 
and, from the hyperconvexity of $(X,\delta)$, we have that there exists $\bar x$ such that $d(\bar x,x)\le r(x)$ for all $x\in X$. So, in particular, $\bar x\in A_i$ for each $i\in I$ and our claim is proved. Now it is immediate that $\bigcap_{i\in I}A_i$ is also in $U$. Therefore we can apply Zorn's lemma on $U$ to deduce that there is a minimal set $A$ in $U$ with respect to set-inclusion. Suppose this minimal set is not a singleton since otherwise it would be a fixed point of $T$.

Since $A$ is minimal and $T(A)\subset A$ we also have that $A=B(T(A))$. Then 
$$
A=\bigcap_{x\in X}\bar B (x,r_x(T(A))).
$$ 
%Moreover, $(X,\delta)$ is bounded, so 
%$$
%\sup \{ \delta (F)\colon F\in \langle X\rangle,\; F\subseteq A\}<\infty.
%$$
Define
$$
d=\sup_{n>1}\frac{\sup_{x_1,\ldots, x_n\in A}\delta(\{x_1,\ldots,x_n\})}{n}.
$$
Since $(X,\delta)$ is bounded, $d$ is attained at a certain $N\in \N$. Fix $N$ as that natural number and $\varepsilon >0$ such that $\varepsilon \le \frac{d}{N}$. Then there is $\{ y_1,y_2,\ldots,y_N\}\subseteq A$ such that
$$
\frac{\delta (\{ y_1,y_2,\ldots,y_N\})}{N}> d-\varepsilon
$$
and so
$$
\delta (\{ y_1,y_2,\ldots,y_N\})> (N-1)d.
$$
Now, from property {\it 2} of diversities, 
$$
\sum_{i=2}^Nd(y_1,y_i)\geq \delta (\{ y_1,y_2,\ldots,y_N\})> (N-1)d
$$
and so, there exists a pair of points $x,z\in \{ y_1,y_2,\ldots,y_N\}$ such that $d(x,z)>d$. Consider these two points fixed from now on.

Make $A^\prime = A\cap \left(\bigcap_{a\in A}\bar B(a,d)\right)$. We will show next that $A^\prime$ is nonempty. First notice that, as any admissible set,
$$
A=\bigcap_{x\in X}\bar B(x,r_x(A))
$$
and consider the function
\begin{equation*}
p(x)= 
\begin{cases} d, & \text{if $x\in A$ and $r_x(A)>d$,}
\\
r_x(A), &\text{otherwise.}
\end{cases}
\end{equation*}
Take $\{ y_1,y_2,\ldots,y_n\}$ a finite subset of $X$. Order these points in such a way that there is $i\in \{0,1,\ldots, n\}$ such that $y_j\in A$ and $p(y_j)=d$ if $j\le i$ and $p(y_j)=r_{y_j}(A)$ if for $j>i$. Then
\begin{align*}
\delta(\{y_1,y_2,\ldots ,y_i\})& \le \sup_{x_1,\ldots,x_i\in A} \delta(\{x_1,x_2,\ldots ,x_i\})\\
&= i\frac{\sup_{x_1,\ldots,x_i\in A} \delta(\{x_1,x_2,\ldots ,x_i\})}{i}\le i\cdot d,
\end{align*}
and, for $j\geq i+1$,
$$
p(y_j)=r_{y_j}(A)\geq d(y_j,y_1)=\delta(\{y_j,y_1\}).
$$
So, in any case, we can proceed as follows (if $i=0$ turn $y_1$ into $y_0$ as any point in $A$)
\begin{align*}
\delta(\{ y_1,y_2,\ldots,y_n\})&\le \delta(\{ y_1,y_2,\ldots,y_i\})+\sum_{j=i+1}^n\delta(\{y_1,y_j\})\\
&\le i\cdot d+\sum_{j=i+1}^n     r_{y_j}(A) =\sum_{k=1}^np(y_k).
\end{align*}
Now it suffices to define 
$$
p(\{x_1,x_2,\ldots,x_n\})=\sum_{k=1}^np(x_k)
$$
and apply the hyperconvexity of diversity $\delta$ to deduce that there exists $\bar x\in X$ such that
$$
d(\bar x, a)\le d
$$ 
for any $a\in A$. In particular we have that $\bar x\in \bigcap_{a\in A}\bar B(a,d)$. Moreover, for any $x\in X$, $d(\bar x,x)\le p(x)\le r_x(A)$ and so $\bar x\in A$. Therefore $\bar x$ is also in $A^\prime$ and, as we wanted to prove, $A^\prime\neq \emptyset$. Now, since $r_x(A)=r_x(T(A))$ we have that $A^\prime$ is $T$-invariant. We will have reached a contradiction with the fact that $A$ is not a singleton if we can show that $ A^\prime\neq A$,  but this follows from the fact that there are two points $x,z\in A$ as above such that $d(x,z)>d$ while ${\rm diam}(A^\prime)\le d$.
\end{proof}

\begin{remark}
As it was pointed out earlier, the above proof follows similar patterns to that one of Theorem 5 in \cite{B} which also follows a parallel argument to the proof of W.A. Kirk for existence of fixed point of nonexpansive mappings in Banach spaces with normal structure in his seminal work \cite{K}. Therefore we are talking here about a very successful line of argument in metric fixed point theory which basically consists in two steps. The first one consists in showing that decreasing families of a certain collection of sets are compact, meaning that their intersections are nonempty and still a set in the same family. Then we apply Zorn's lemma to obtain a minimal set and, as a second step, a normal structure argument is applied to show that that minimal set must be a singleton. Notice that in our proof step one only requires that the induced metric space is bounded and not the diversity itself. Step two, however, requires that the diversity itself were bounded. From the example exhibited in Theorem \ref{th1} we can conclude that this condition is necessary. 
\end{remark}

The following corollary is an immediate consequence of Propositions \ref{admissible} and \ref{prophyper}, and Theorem \ref{fixedpoint}.

\begin{corollary}\label{admifpp}
Let $(X,d)$ be a metric space induced by a bounded hyperconvex diversity, then any admissible subset of $X$ or any set as those given in Proposition \ref{prophyper} has the fixed point property for nonexpansive self-mappings.
\end{corollary}

\begin{remark}
Theorem \ref{fixedpoint} can be seen as a generalization of the fixed point theorem for hyperconvex spaces. This follows as a direct consequence of Theorem 4.3 in \cite{BT} since it is clear after it that the diameter diversity of a hyperconvex metric space is a hyperconvex diversity. Therefore any bounded hyperconvex metric space is the induced metric space of a bounded hyperconvex diversity.
\end{remark}

A very relevant consequence of the approach of J.P. Baillon to the fixed point property for nonexpansive mappings was that decreasing families of bounded hyperconvex metric spaces need to have nonempty hyperconvex intersection. In the proof of Theorem \ref{th1} we saw that this is not the case for hyperconvex diversities regardless the induces metric space is bounded or not. Next, following the argument of Baillon in \cite{B}, we will show that such intersections are also nonempty when decreasing diversities are supposed to be bounded.

\begin{theorem}\label{intersection}
Let $(X,\delta)$ be a bounded hyperconvex diversity and $(H_i)_{i\in I}$ a directed family of subsets of $X$ such that $(H_i,\delta)$ is hyperconvex for each $i\in I$. Then the intersection of this family of sets is nonempty.
\end{theorem}

\begin{proof}
We include Baillon's original scheme of the proof for completeness. Notice that some technical facts need to be solved in a different way.

Step 1. We consider the product space
$$
H =\prod_{i\in I}H_i
$$
where $H_i \subset H_j, \; j \leq i$. For any subset $E$ of $X$ define:
$$
B_{i}(E)=\bigcap_{x\in H_i}\bar B(x,r_x(E)).
$$
Then we consider the class sets
$$
{\cal U}=\{ A=\prod_{i\in I}A_i\colon A\neq\emptyset,\; A_i\subseteq H_i,\; B_i(A_i)\cap H_i
=A_i,\; A_j\subseteq A_i\text{ if $j\geq i$}\}.
$$

Step 2. To show that there is a~minimal element of $\cal U$ we apply the same methods as in proof of Theorem \ref{fixedpoint} for each $i\in I$, therefore we can fix $\Pi_I A_i$ as a~minimal element in $H$.

Step 3. We notice now that, from minimality of $A$, there is no pair $(j,i)$, $(j \geq i)$  such that 
$$
A_i \neq A_i \cap\bigcap_{x\in H_j}B(x,r_x(A_j))
$$
since it must be the case that
$$
A_i \supset A_i \cap\bigcap_{x\in H_j}B(x,r_x(A_j))=A_i^\prime,
$$
contradicting the minimality of $A$.

Step 4. From Step 3 we directly obtain that:
\begin{equation}\label{promienie}
r_x(A_i)=r_x(A_j), \; \text{ for $x\in H_j$ and any $(j,i)$ with $j\geq i$}.
\end{equation}

Step 5. We claim that $A_i$ is a singleton. From \eqref{promienie} we know that if $A_i$ is not a singleton then $A_j$ is not a singleton for any $j \geq i$. Consider now two subsets $A$ and $B$ of $X$.
Since $(X,\delta)$ is bounded then the induced diversity on these sets is also bounded and moreover, if $A\subset B$, then we have that
$$
d_A:=\sup\limits_{n>1}\dfrac{\sup_{x_1,\ldots,x_n\in A}\delta(\{x_1,\ldots,x_n\})}{n}
$$
is not greater than
$$
d_B:=\sup\limits_{n>1}\dfrac{\sup_{x_1,\ldots,x_n\in B}\delta(\{x_1,\ldots,x_n\})}{n}.
$$
%Indeed, it is trivial. 
%$$
%\dfrac{\sup_{x_1,\ldots,x_n\in A}\delta(\{x_1,\ldots,x_n\})}{n} \leq \dfrac{\sup_{x_1,\ldots,x_n\in B}\delta(\{x_1,\ldots,x_n\})}{n} \leq d_B
%$$
%for any natural $n$.

Now, for $i\in I$, consider $A_i^\prime$ given by
$$
A_i^\prime=A_i\cap\bigcap_{x\in A_i}\bar{B}(x,d_{A_i}).
$$
Therefore for each set $A_i$ we have a new one $A_i^\prime$ in $H_i$. We in fact can show, in a similar way as in the proof of Theorem \ref{fixedpoint}, that $A_i^\prime$ is contained but not equal to that $A_i$. Therefore we only need to show that $A_j^\prime\subset A_i^\prime$ if $i\leq j$ to contradict the minimality of $\Pi_I A_i$.

Indeed, consider $x\in A_j^\prime$. Obviously $x\in A_i$. On the other hand, $x\in H_j$, so $r_x(A_j)=r_x(A_i)$. At the same time $x\in \bigcap_{y\in A_j}\bar B(y,d_{A_j})$, so 
$$
r_x(A_i)=r_x(A_j) \leq d_{A_j} \leq d_{A_i}
$$ 
and finally 
$$
x\in \bigcap_{z\in A_i}\bar B(z,d_{A_i}),
$$
which completes the proof.\end{proof}

In \cite{B} a bit more was proved, it was proved that the intersection of bounded and decreasing collections of hyperconvex metric spaces was hyperconvex too. Next we give this result for diversities.

\begin{corollary}
Under the conditions of the previous theorem, the intersection set with the induced diversity is a hyperconvex diversity.
\end{corollary}

\begin{proof}
Let $Y$ be the intersection of the sets $H_i$. Consider $r\colon \langle Y\rangle \to \R$ satisfying the hyperconvexity condition \eqref{rhyper}. From Lemma \ref{hyperholds}, and boundedness conditions, we can consider $r$ defined on the whole $\langle H_1\rangle$ and so on any of the sets $\langle H_i\rangle$. For each $i$ consider the set $Y_i$ given by $r$, $H_i$ and $Y$ after Propositon \ref{prophyper} (notice that $r$ is the same for all $H_i$). Then $(Y_i,\delta)$ is a decreasing collection of hyperconvex diversities as in Theorem \ref{intersection}. Therefore there is $y$ in all of these sets and so $y\in Y$ is such that $\delta (A\cup \{y\})\le r(A)$ for each $A\in \langle Y\rangle$. 
\end{proof}

{\bf Acknowledgements:}

The first author wishes to acknowledge The University of Seville and the Department of Mathematical Analysis for its help to visit this institution and to Grant MTM2009-10696-C02-01 for financial support. A big part of this work was produced during that visit.

Rafa Esp\' inola was supported by DGES, Grant MTM2012-34847C02-01 and Junta de Andaluc\'ia, Grant FQM-127.

{\small
\[\begin{tabular}{ll@{\hspace{1cm}}l}
\text{Bo\.zena Pi\c{a}tek}& & \text{Rafa Esp\'{\i}nola} \\
\text{\small Institute of Mathematics} & & \text{\small Dpto. de An\'alisis Matem\'atico} \\
\text{\small Silesian University of Technology}& & \text{\small Universidad de Sevilla, P.O.Box 1160} \\
\text{\small 44-100 Gliwice, Poland}& & \text{\small 41080-Sevilla, Spain} \\
\text{\small email: {\tt Bozena.Piatek@polsl.pl}}& & \text{\small email: {\tt espinola@us.es }}\\
\text{\small }& & \text{\small fax number: {\tt (34) 954 557 972}}
\end{tabular}\]
}

\end{document}